\documentclass[12pt,reqno]{amsart}
\usepackage{amsmath}
\usepackage{amssymb}
\usepackage{amsthm}
\usepackage[left=2.7cm,top=2.7cm,right=2.7cm,bottom=2.7cm]{geometry}
\usepackage{enumerate}
\usepackage[usenames]{color}
\usepackage{graphicx}
\usepackage{hyperref}

\newtheorem{theorem}{Theorem}[section]
\newtheorem{lemma}[theorem]{Lemma}

\newtheorem{corollary}[theorem]{Corollary}

\newcommand{\cal}{\mathcal}

\newcommand{\brac}[1]{\left(#1\right)}

\newcommand{\bfrac}[2]{\left(\frac{#1}{#2}\right)}

\def\cE{{\cal E}}

\newcommand{\beq}[2]{\begin{equation}\label{#1}#2\end{equation}}

\newcommand{\set}[1]{\left\{#1\right\}}

\def\E{\mbox{{\bf E}}}

\def\Pr{\mbox{{\bf Pr}}}

\newcommand{\ignore}[1]{}

\def\dist{\;\text{dist}}

 \def\b{\beta}  
\def\e{\varepsilon}    
  
 \def\th{\theta}    
   
\def\r{\rho}   
 \def\om{\omega}

\def\E{\mbox{{\bf E}}}

\def\Pr{\mbox{{\bf Pr}}}

\def\dist{\;\text{dist}}
\def\la{\zeta}

\newcommand{\diam}{\textrm{diam}}
\def\deg{\text{deg}}
\def\codeg{\text{codeg}}
\begin{document}

\title{A note on the localization number of random graphs: diameter two case}

\author{Andrzej Dudek}
\address{Department of Mathematics, Western Michigan University, Kalamazoo, MI}
\email{\tt andrzej.dudek@wmich.edu}
\thanks{The first author was supported in part by a grant from the Simons Foundation (522400, AD)}

\author{Alan Frieze}
\address{Department of Mathematical Sciences, Carnegie Mellon University, Pittsburgh, PA}
\email{\tt alan@random.math.cmu.edu}
\thanks{The second author was supported in part by NSF grant DMS1661063}

\author{Wesley Pegden}
\address{Department of Mathematical Sciences, Carnegie Mellon University, Pittsburgh, PA}
\email{\tt wes@math.cmu.edu}
\thanks{The third author was supported in part by NSF grant DMS136313}


\begin{abstract}
We study the localization game on dense random graphs. In this game, a {\em cop} $x$ tries to locate a {\em robber} $y$ by asking for the graph distance of $y$ from every vertex in a sequence of sets $W_1,W_2,\ldots,W_\ell$. We prove high probability upper and lower bounds for the minimum size of each $W_i$ that will guarantee that $x$ will be able to locate $y$. 
\end{abstract}

\maketitle

\section{Introduction}
In this paper we consider the following {\em Localization Game} related to the well studied {\em Cops and Robbers} game; see Bonato and Nowakowski \cite{BN} for a survey on this game. A robber is located at a vertex $v$ of a graph $G$. In each round, a cop can ask for the graph distance between $v$ and vertices $W=\set{w_1,w_2,\ldots,w_k}$, where a new set of vertices $W$ can be chosen at the start of each round. The cop wins immediately if {\em the $W$-signature} of $v$, i.e. the set of distances,  $\dist(v,w_i)$, $i=1,2,\ldots,k$ is sufficient to determine $v$. Otherwise, the robber will move to a neighbor of $v$ and the cop will try again with a (possibly) different {\em test set}~$W$. Given $G$, the {\em localization number} $\la(G)$ is the minimum $k$ so that the cop can eventually locate the robber, that means, the cop determines the exact location of the robber from the test sets of size~$k$. This game was introduced by Bosek et al. \cite{Bosek}, who studied the localization game on geometric and planar graphs, and also independently, by Haslegrave et al.~\cite{HJK}. For some other related results see~\cite{West, S, S2}.

\section{Results}
The localization number is closely related to the {\em metric dimension} $\b(G)$. This is the smallest integer $k$ such that the cop can always win the game in {\em one} round. Clearly, $\la(G)\leq \b(G)$.

In this note we will study the localization number of the random graph $G_{n,p}$ with diameter two. Here and throughout the whole paper $\omega = \omega(n)=o(\log n)$ denotes a function tending arbitrarily slowly to infinity with~$n$. We will also use the notation
$$q=1-p\text{ and }\r=p^2+q^2.$$
We write $A_n \lesssim B_n$ to mean that $A_n \le (1+o(1))B_n$ as $n$ tends to infinity. We further write $A_n\approx B_n$ if $A_n =(1+o(1))B_n$ as $n$ tends to infinity. Finally, we say that an event $\cE_n$ occurs {\em asymptotically almost surely}, or a.a.s. for brevity, if $\lim_{n\rightarrow\infty}\Pr(\cE_n)=1$.

The metric dimension of $G_{n,p}$ was studied by Bollob\'as et al.~\cite{BMP}. If we specialize their result to large $p$ then it can be expressed as:
\begin{theorem}[\cite{BMP}]\label{thm:metric}
Suppose that
\[
\bfrac{2\log n+\om}{n}^{1/2}\leq p \le 1 - \frac{3 \log \log n}{\log n}.
\]
Then, 
\beq{BMP0}{
\frac{2\log np}{\log 1/\r}\lesssim \b(G_{n,p}) \lesssim \frac{2\log n}{\log 1/\r}\ a.a.s..
}
\end{theorem}
Note that the upper and lower bounds in \eqref{BMP0} are asymptotically equal if $p\geq n^{-o(1)}$.

It is well-known (see, e.g.,~\cite{FK}) that if $np^2\geq 2\log n +\om$, then a.a.s.\ $\diam(G_{n,p})\leq 2$. We will condition on the diameter satisfying this. 
Graphs with diameter~2 enable some simplifications. Indeed, if a vertex $v$ has $W$-signature $\{d_1,\dots,d_k\}$, where $W=\{w_1,\dots,w_k\}$, where $d_i = \dist(v,w_i)$, then 
\[
d_i = 
\begin{cases}
1 & \text{iff } \{v,w_i\} \in E\\
2 & \text{iff } \{v,w_i\} \notin E.
\end{cases}
\]
Consequently, the probability that two vertices $u$ and $v$ in $G_{n,p}$ have the same $W$-signature, $W = \{w_1,\dots,w_k\}$, such that $u,v\notin W$ is equal to
\[
\prod_{i=1}^k \Pr(u,v\in N(w_i) \text{\ \,or\ \,} u,v\notin N(w_i)) = \r^k.
\]

The upper bound on $p$ in the below theorem is determined by a result of~\cite{BMP} about the metric dimension of~$G_{n,p}$.

\begin{theorem}\label{thm:diam2}
Let 
\[
\bfrac{2\log n + \omega}{n}^{1/2} \le p \le 1 - \frac{3\log \log n}{\log n} \quad \text{ and } \quad \eta=\frac{\log(1/p)}{\log n}
\]
and let $c$ be a positive constant such that
\[
0< c < \min\left\{ \frac{1}{2}\left( \frac{\log n - 3\log\log n}{\log1/p}  - 1\right), 1 \right\}.
\]
Then, a.a.s.
$$\brac{1 - 2\eta-\frac{4\log\log n}{\log n}} \frac{2\log n}{\log 1/\r}\le \zeta(G_{n,p})\le 
(1 - c\eta) \frac{2\log n}{\log 1/\r}.$$
\end{theorem}
\subsection{Observations about Theorem \ref{thm:diam2}}\ \\
First observe that if $p\ge \frac{\log n}{n^{1/3}}$, then 
\[
\frac{1}{2}\left( \frac{\log n - 3\log\log n}{\log1/p}  - 1\right) \ge 1
\]
and so $c$ can be any positive constant less than 1. Furthermore, for any $p \ge \bfrac{2\log n+\omega}{n}^{1/2}$ we have
\[
\frac{1}{2}\left( \frac{\log n - 3\log\log n}{\log1/p}  - 1\right) \ge
\frac{1}{2}\left( \frac{\log n - 3\log\log n}{\frac{1}{2}(\log n - \log (2\log n + \omega))}  - 1\right) =\frac{1}{2}-o(1).
\]
Hence, we can always take $c\geq \frac{1}{2}-o(1)$.

\vspace{.2in}
If $p=1/n^{\alpha}$ for some constant $0 < \alpha <1/2$, then, 
\[
\eta = \alpha \quad \text{ and } \quad c \leq
\begin{cases}
1-o(1) & \text{ if } 0<\alpha < \frac{1}{3}\\
\frac{1}{2\alpha} -\frac{1}{2}-o(1) & \text{ otherwise.}
\end{cases}
\]
Moreover,
\[
\r = 1-2p+2p^2 \text{ and so }\log1/\r = 2p +O(p^2)\approx \frac{2}{n^{\alpha}}.
\]
Hence, Theorem~\ref{thm:diam2} implies the following corollary.
\begin{corollary}
Let $p = 1/n^{\alpha}$, where $0 < \alpha <1/2$ is constant. Then, a.a.s.\
$$
(1 - 2\alpha) n^{\alpha} \log n\lesssim \zeta(G_{n,p}) \lesssim 
\begin{cases}
(1 - \alpha) n^{\alpha} \log n & \text{ if } 0<\alpha < \frac{1}{3}\\
\left(\frac{1+\alpha}{2}\right) n^{\alpha} \log n & \text{ otherwise}.
\end{cases}
$$
\end{corollary}
\noindent
Notice that for $0<\alpha < \frac{1}{3}$ the upper bound on $\zeta(G_{n,p})$ equals the lower bound from Theorem~\ref{thm:metric}. Therefore, it is plausible to conjecture that $\zeta(G_{n,p}) < \beta(G_{n,p})$.

Now observe that if $p = n^{-1/\omega}$, then
\[
2\eta = \frac{2\log(1/p)}{\log n} = \frac{2}{\omega} = o(1).
\]
Thus, Theorem~\ref{thm:diam2} implies:
\begin{corollary}
Let $p = n^{-1/\omega}$. Then,
\[
\zeta(G_{n,p}) \approx \frac{2\log n}{\log1/\r}.
\]
\end{corollary}
\noindent
Clearly, this also holds for any constant~$p$. In particular, for $p=1/2$, we get:
\begin{corollary}
For almost all graphs~$G$ we have
\[
\zeta(G)\approx \frac{2\log n}{\log 2} = 2\log_2(n).
\]
\end{corollary}

\subsection{Proof of Theorem~\ref{thm:diam2} -- lower bound}\ \\

Since we will deal with ``mostly independent'' random variables, we will use the following form of Suen's inequality (see, e.g. \cite{JLR}).
\begin{theorem}[Suen's Inequality]\label{ineq:suen}
Let $\th_i,i\in I$ be indicator random variables which take value 1 with probability~$p_i$. Let $L$ be a dependency graph. Let $X = \sum_{i\in I}\th_i$, and $\mu = \E(X) = \sum_{i\in I} p_i$. Moreover, write $i\sim j$ if $ij \in E(L)$, and let $\Delta = \frac{1}{2}\sum_{i\sim j} \E(\th_i\th_j)$ and $\delta = \max_i\sum_{j\sim i} p_j$. Then,
$$\Pr(X = 0) \le \exp\left\{-\min\set{ \frac{\mu^2}{8\Delta}, \frac{\mu}{2}, \frac{\mu}{6\delta} } \right\}.$$
\end{theorem}
We will also use the following simple fact.
\begin{lemma}\label{claim:32}
Let $0<p<1$ and $p+q=1$. Then, 
\[
\frac{\log(p^3+q^3)}{\log\r} \ge \frac{3}{2}.
\]
\end{lemma}
\begin{proof}
This inequality is equivalent to
\[
\log(p^3+q^3)^2 \le \log(p^2+q^2)^3
\]
and so to
\[
(p^3+q^3)^2 \le (p^2+q^2)^3.
\]
The latter is equivalent to 
\[
2p^3q^3 \le 3p^4q^2 + 3 p^2q^4 = 3p^2q^2 (p^2+q^2) = 3p^2q^2 (1 - 2pq )
\]
and consequently to
\[
2pq \le 3(1-2pq)
\]
which is equivalent to 
\[
pq \le \frac{3}{8}.
\]
But this is always true since $pq\leq \frac14$.
\end{proof}

The lower bound in Theorem~\ref{thm:diam2} will follow from the following result.

\begin{lemma}\label{lem:diam2:low}
Let 
\[
\frac{\log^2n}{n^{1/2}} < p \le 1-\frac{1}{\log n}
\quad \text{ and } \quad
\e = \frac{2\log\left( \frac{\log^2 n}{p}\right)}{\log n} \quad \text{ and } \quad k = \frac{2(1-\e) \log n}{\log1/\r}.
\]
Then a.a.s.,
$$\zeta(G_{n,p}) \ge k.$$
\end{lemma}
First observe that $\e=2\eta+\frac{4\log\log n}{\log n}$ and so the lower bound in Theorem~\ref{thm:diam2} holds.
\begin{proof}
For a fixed vertex $u$ and $k$-set $S$ let $X_{u,S}$ count the number of unordered pairs $w,v\in N(u)$ with the same signature induced by  $S$. We prove that the probability that there is a vertex $u$ and a $k$-set~$S$ such that $X_{u,S}=0$ is $o(1)$. Consequently, this will imply that a.a.s.\ for every vertex $u$ and $k$-set $S$ there are at least two neighbors of $u$ with the same signature in $S$. Hence, a.a.s.\ the localization number is at least~$k$.

Clearly, 
\begin{align*}
\mu=\E(X_{u,S}) = \binom{n - k-1}{2} \r^k p^2 &\ge \frac{p^2}{4}\exp\{k\log \r +2\log n\}\\
& =\frac{p^2}{4}\exp\{-2(1-\e) \log n +2\log n\} = \frac{p^2}{4}n^{2\e}.
\end{align*}
Furthermore, since every triple of vertices in $N(u)$ with the same signature contributes three unordered pairs of variables to $\Delta$, we get
\begin{align*}
\Delta &\le  3\binom{n}{3} (p^3+q^3)^k p^3\\
&\leq \frac{p^3}{2} \exp\left\{ k\log(p^3+q^3) + 3\log n\right\}\\
&= \frac{p^3}{2} \exp\left\{ -2(1-\e)(\log n) \frac{\log(p^3+q^3)}{\log \r} + 3\log n \right\}.
\end{align*}
Now, by Lemma~\ref{claim:32},
\[
\Delta \le \frac{p^3}{2} \exp\left\{ -2(1-\e)(\log n) \cdot \frac{3}{2} + 3\log n \right\}
= \frac{p^3}{2} n^{3\e}.
\]
Similarly 
\begin{align*}
\delta \le 2n \r^k p^2
= 2p^2 \exp\left(k\log\r + \log n\right)
=2p^2 n^{-1+2\e}.
\end{align*}

Thus,
\[
\frac{\mu^2}{8\Delta} \ge \frac{1}{64} pn^{\e},
\quad \frac{\mu}{2} \ge \frac{1}{8} (pn^{\e})^2
\quad \text{ and } \quad \frac{\mu}{6\delta} \ge \frac{1}{48} n.
\]
Since $0<\e<1$ and $pn^\e \to \infty$ (due to our choice of $\e$) the lower bound in the first inequality is the smallest. Hence, by Theorem \ref{ineq:suen},
\[
\Pr(X_{u,S} = 0) \le \exp\left\{ - \frac{1}{64} pn^{\e}\right\}.
\]

Now we use the union bound to show that the probability that there is a vertex $u$ and a $k$-set~$S$ such that $X_{u,S}=0$ is $o(1)$. Indeed, this probability is at most
\begin{equation}\label{eq:union_bound_XuS}
n\binom{n}{k}  \exp\left\{ - \frac{1}{64} pn^{\e}\right\}
\le \exp\left\{ (k+1)\log n - \frac{1}{64} pn^{\e}\right\}.
\end{equation}
Now observe that $\r = (p+q)^2 - 2pq = 1-2pq$ and so
\[
k = \frac{2(1-\e) \log n}{\log1/\r} = -\frac{2(1-\e) \log n}{\log(1-2pq)} \le -\frac{2 \log n}{\log(1-2pq)}.
\]
Since $1-x \le e^{-x}$ and $2pq < 1$, we get that
\[
k\log n \le\frac{(\log n)^2}{pq}.
\]
Furthermore, since by assumption $p \le 1-\frac{1}{\log n}$, we obtain $q\ge \frac{1}{\log n}$ and so
\[
k\log n \le \frac{(\log n)^3}{p}.
\]
Also 
\[
pn^\e = pe^{\e\log n} = \frac{(\log n)^4}{p}.
\]
Thus, the exponent in~\eqref{eq:union_bound_XuS} tends to $-\infty$. This completes the proof of Lemma~\ref{lem:diam2:low}. 
\end{proof}

\subsection{Proof of Theorem~\ref{thm:diam2} -- upper bound}\ \\
Let $\deg(v)$ denote the degree of vertex $v$ in $G_{n,p}$ and let $\codeg(v,w)$ denote the co-degree of vertices $v,w$ in $G_{n,p}$. We observe next that the Chernoff bounds imply that a.a.s.
\begin{align}
\deg(v)&=np+O((np\log n)^{1/2})\text{ for all }v\in [n].\label{degrees}\\
\codeg(v,w)&=np^2+O((np^2\log n)^{1/2})\text{ for all }v\in [n].\label{co-degrees}
\end{align}
\begin{lemma}\label{lem:ub}
\noindent
\begin{enumerate}[(i)]
\item\label{lem:ub:i} Let 
\[
e^{-\frac{\log n}{\omega}} \le p \le 1 - \frac{3\log \log n}{\log n}.
\]
Then, a.a.s.\ 
$$\zeta(G_{n,p}) \lesssim \frac{2\log n}{\log1/\r}.$$ 
\item\label{lem:ub:ii} Let
\[
\bfrac{2\log n+\omega}{n}^{1/2} \le p \le e^{-\Omega(\log n)}
\quad \text{ and } \quad
\eta = \frac{\log1/p}{\log n} \quad \text{ and } \quad k = \frac{2(1-c\eta) \log n}{\log1/\r},
\]
where
\[
0< c < \min\left\{ \frac{1}{2}\left( \frac{\log n - 3\log\log n}{\log1/p}  - 1\right), 1 \right\}.
\]
Then, a.a.s. 
$$\zeta(G_{n,p}) \le k.$$ 
\end{enumerate}
\end{lemma}
\begin{proof}
Part~\eqref{lem:ub:i} follows immediately from Theorem~\ref{thm:metric}. 
\medskip

We now prove~\eqref{lem:ub:ii}. Equations \eqref{degrees} and \eqref{co-degrees} plus our bound of two on the diameter are all we need for this. So the analysis works for any graph satisfying these conditions.  Let $S_1$ be a randomly chosen $k$-subset of $V$ and let $X_1$ be the number of pairs with the same signature in $S_1$. Then, if $$D(v,w)=(N(v)\setminus N(w))\cup (N(w)\setminus N(v))$$
 for $v,w\in [n]$ then
\begin{align}
\E(X_1) &=\sum_{v\neq w}\Pr((N(v)\cap S_1)=(N(w)\cap S_1))\nonumber\\
&=\sum_{v\neq w}\Pr(S_1\cap D(v,w)=\emptyset)\nonumber\\
&\leq n^2\brac{1-2p(1-p)\brac{1+O\bfrac{\log^{1/2}n}{n^{1/2}}}}^k\label{xx}\\
&=n^2\rho^k\brac{1+O\bfrac{k\log^{1/2}n}{n^{1/2}}}\label{xxx}\\
&=(1+o(1))n^{2c\eta}.\label{first}
\end{align}
and by the Markov inequality we have $X_1 \le \om n^{2c\eta}$ a.a.s.. (Going from \eqref{xx} to \eqref{xxx} uses the trivial identity $1-a(1-\e)=(1-a)\brac{1+\frac{a\e}{1-a}}$.) Thus, the set $R$ of vertices with exactly the same signature in $S$ as the robber is a.a.s. of size at most $\om^{1/2}n^{c\eta}$. Let $T_2$ consist of $R$ and the set of neighbors of $R$. The robber can move to somewhere in $T_2$. Clearly, $|T_2| \le 2\om^{1/2} n^{c\eta} pn$ a.a.s..

Now let $S_2$ be another random $k$-subset of $V$, chosen independently of $S_1$. Let $X_2$ be the number of pairs of vertices from $T_2$ with the same signature in $S_2$. Arguing as for \eqref{first}, we see that
\begin{multline*}
\E(X_2) \le (2\om^{1/2} n^{c\eta} pn)^2 \r^{k} \brac{1+O\bfrac{k\log^{1/2}n}{n^{1/2}p}}\\=(1+o(1)) (2\om^{1/2} p)^2 \exp((2+2c\eta)(\log n) + k\log\r) = (4+o(1))\om p^2 n^{4c\eta} 
\end{multline*}
and by the Markov inequality we get that a.a.s we have $X_2 \le \om^2 p^2 n^{4c\eta}$. Thus, the number of vertices with exactly the same signature as the robber in $S_2$ is at most $\om p n^{2c\eta}$. Let $T_3$ consist of these vertices together with their neighbors. Clearly, $|T_3| \leq2\om  p^2 n^{2c\eta+1}$.

We proceed inductively. Assume that $|T_i| \le 2(\om^{1/2} p)^{i-1} n^{(i-1)c\eta+1}$. Now, arguing as above with another independently chosen $k$-set $S_{i+1}$, we have
\[
\E(X_{i+1}) \le (2+o(1))((\om^{1/2} p)^{i-1} n^{(i-1)c\eta+1})^2 \r^{k} = (2+o(1))(\om^{1/2} p)^{2(i-1)} n^{2ic\eta}
\]
and so by the Markov inequality, 
\beq{Xi}{
X_{i+1}\le \om(\om^{1/2} p)^{2(i-1)} n^{2ic\eta}\text{ a.a.s..}
}
Thus, the number of vertices with exactly the same signature in $S_{i+1}$ is at most\linebreak $\om^{1/2}(\om^{1/2} p)^{i-1} n^{ic\eta}$. Hence,
\[
|T_{i+1}| \le 2\om^{1/2}(\om^{1/2} p)^{i-1} n^{ic\eta} pn = 2(\om^{1/2} p)^{i} n^{ic\eta+1},
\]
completing the induction.

After $\ell$ rounds we get that with probability at least $1-\ell\om^{-1}$ we have, using \eqref{Xi},
\begin{align}\label{eq:Xl}
|X_\ell| &\le \om(\om^{1/2} p)^{2(\ell-2)} n^{2(\ell-1)c\eta} =\om^{\ell-1} \exp\set{2(\ell-2)\log p + 2(\ell-1)c\eta\log n)} \notag \\
&= \om^{\ell-1} \exp\set{-2(\ell-2-c(\ell-1))\log(1/p)}.
\end{align}
Clearly, \eqref{eq:Xl} is $o(1)$ for sufficiently large constant~$\ell$, since by assumption $\log(1/p) = \Omega(\log n)$.
\end{proof}
\section{Summary}
We have separated the localization value $\la(G_{n,p})$ from the metric dimension $\b(G_{n,p})$ in the range where the diameter of $G_{n,p}$ is two a.a.s.. 
It would be interesting to continue the analysis in the range of $p$ for which the diameter of $G_{n,p}$ is at least 3. It would also be of interest to examine the localization game on random regular graphs.

\medskip

\textbf{Acknowledgment} We are grateful to all referees for their detailed comments on an earlier version of this paper.

\end{document}